\documentclass{amsart}
\usepackage[T1]{fontenc}
\usepackage[utf8]{inputenc}
\usepackage{lmodern}
\usepackage{amsmath,amssymb,amsthm,bbm}
\usepackage[english]{babel}

\DeclareMathOperator{\Idem}{Idem}

\newcommand{\shape}[1]{\mathcal{S}(#1)}

\newcommand{\communiquelambda}{\stackrel{\lambda}{\leftrightarrow}}
\newcommand{\communiquelambdap}{\stackrel{\lambda'}{\leftrightarrow}}

\newcommand{\N}{\mathbb N}

\newcommand{\Zd}{\mathbb{Z}^d}

\newcommand{\R}{\mathbb{R}}
\newcommand{\Rd}{\mathbb{R}^d}
\renewcommand{\P}{\mathbb{P}}
\newcommand{\E}{\mathbb{E}}
\newcommand{\Ed}{\mathbb{E}^d}
\newcommand{\eps}{\varepsilon}

\newcommand{\Ebarre}{\overline{\mathbb{E}}}
\newcommand{\Pbarre}{\overline{\mathbb{P}}}

\renewcommand{\epsilon}{\varepsilon}
\renewcommand{\phi}{\varphi}
\renewcommand{\limsup}{\overline{\lim}}

\newcommand{\ie}{\emph{i.e. }}

\newcommand{\miniop}[3]{%
\renewcommand{\arraystretch}{0.6}
\begin{array}{c}
{\scriptstyle #1}\\
#2\\
{\scriptstyle #3}
\end{array}
\renewcommand{\arraystretch}{1}}

\newcommand{\1}{1\hspace{-1.3mm}1}

\newtheorem{theorem}{Theorem}
\newtheorem{lemme}[theorem]{Lemma}

\newtheorem{coro}[theorem]{Corollary}
\newtheorem{prop}[theorem]{Proposition}

\def\restriction#1#2{\mathchoice
              {\setbox1\hbox{${\displaystyle #1}_{\scriptstyle #2}$}
              \restrictionaux{#1}{#2}}
              {\setbox1\hbox{${\textstyle #1}_{\scriptstyle #2}$}
              \restrictionaux{#1}{#2}}
              {\setbox1\hbox{${\scriptstyle #1}_{\scriptscriptstyle #2}$}
              \restrictionaux{#1}{#2}}
              {\setbox1\hbox{${\scriptscriptstyle #1}_{\scriptscriptstyle #2}$}
              \restrictionaux{#1}{#2}}}
\def\restrictionaux#1#2{{#1\,\smash{\vrule height .8\ht1 depth .85\dp1}}_{\,#2}}

\begin{document}

\title[Continuity of the asymptotic shape]{Continuity of the asymptotic shape of the supercritical contact process}

{
\author{Olivier Garet}
\address{
Universit\'e de Lorraine, Institut \'Elie Cartan de Lorraine, UMR 7502, Vandoeuvre-l{\`e}s-Nancy, F-54506, France\\
\and \\
CNRS, Institut \'Elie Cartan de Lorraine, UMR 7502, Vandoeuvre-l{\`e}s-Nancy, F-54506, France\\
}
\email{Olivier.Garet@univ-lorraine.fr}

\author{R{\'e}gine Marchand}
\address{
Universit\'e de Lorraine, Institut \'Elie Cartan de Lorraine, UMR 7502, Vandoeuvre-l{\`e}s-Nancy, F-54506, France\\
\and \\
CNRS, Institut \'Elie Cartan de Lorraine, UMR 7502, Vandoeuvre-l{\`e}s-Nancy, F-54506, France\\
}
\email{Regine.Marchand@univ-lorraine.fr}

\author{Marie Th{\'e}ret}
\address{LPMA, Universit\'e Paris Diderot, 5 rue Thomas Mann, 75205 Paris Cedex 13, France}
\email{marie.theret@univ-paris-diderot.fr}
}

\def\motsclefs{contact process, shape theorem, continuity.}
\subjclass[2000]{60K35, 82B43.}

\begin{abstract}
We prove the continuity of the shape governing the asymptotic growth of the supercritical contact process in $\Zd$, with respect to the infection parameter. The proof is valid in any dimension $d \ge 1$.
\end{abstract}

\maketitle


\section{Introduction}

The contact process is a famous interacting particle system modelling the spread of an infection on the sites of $\Zd$. The evolution depends on a fixed parameter $\lambda \in (0, +\infty)$ and is as follows: at each moment, an infected site becomes healthy at rate $1$ while a healthy site becomes infected at a rate equal to $\lambda$ times the number of its infected neighbors.
There exists a critical value $\lambda_c(\Zd)$ such that the infection, starting from the origin, infinitely expands with positive probability if and only if $\lambda>\lambda_c(\Zd)$.

Durrett and Griffeath~\cite{MR656515} proved that when the contact process on $\Zd$ starting from the origin survives, the set of sites occupied before time $t$ satisfies an asymptotic shape theorem, as in first-passage percolation. In~\cite{GM-contact}, two of us extended this result to the case of the contact process in a random environment. 
The shape theorem can be stated as follows: provided that $\lambda>\lambda_c(\Zd)$, there exists a norm $\mu_{\lambda}$ on $\Rd$ such that the set $H_t$ of points already infected before time $t$ satisfies:
$$\Pbarre_{\lambda}\left(\exists T>0:  \quad  t\ge T \; \Longrightarrow \;  (1-\epsilon)t\shape{\lambda}\subset \tilde{H_t} \subset (1+\epsilon)t\shape{\lambda}\right)=1,$$
where $\tilde{H_t}=H_t+[0,1]^d$, $\shape{\lambda}$ is the unit ball for $\mu_{\lambda}$ and $\Pbarre_{\lambda}$ is the law of the contact process with parameter $\lambda$, starting from the origin and conditioned to survive. The growth of the contact process is thus asymptotically linear in time, and governed by the shape $\shape{\lambda}$. 

The aim of this note is to prove the continuity of the map 
$\lambda\mapsto \shape{\lambda}$. More precisely, we prove the following result:
\begin{theorem} 
For every  $\lambda>\lambda_c(\Zd)$,
$\displaystyle \lim_{\lambda'\to \lambda} \sup_{x \in \mathbb S^{d-1}}|\mu_{\lambda'}(x)-\mu_{\lambda}(x)|=0,$ \\
where $\displaystyle \mathbb S^{d-1}=\left\{x=(x_i)_{1 \le i\le d} \in \Rd: \; \|x\|_1=\sum_{i=1}^d |x_i|=1\right\}$.
\end{theorem}
It is then easy to deduce the following continuity for the asymptotic shape.
Denote by $d_H$ the Hausdorff distance between non-empty compact sets in $\Rd$. For every  $\lambda>\lambda_c(\Zd)$,
$$\lim_{\lambda'\to \lambda} d_H(\shape{\lambda'},\shape{\lambda})=0.$$

Continuity properties for asymptotic shapes in random growth models have already been investigated.
In first passage percolation, perhaps the most famous random growth model, 
Cox and Kesten \cite{MR588407,MR633228,kesten} proved that the time constant is continuous with respect to the distribution of the passage-time of an edge. 
In a forthcoming paper, Garet, Marchand, Procaccia and Th\'eret~\cite{GMPT} extend their result  to the case of possibly infinite passage times by renormalization techniques.

In these two cases, thanks to a good subadditivity property, the quantity whose continuity is studied appears as an infimum of a decreasing sequence of continuous functions, which gives quite easily one half of the continuity.

Because of the possibility of extinction of the contact process, the subadditivity properties are not so obvious and we thus use the essential hitting time presented in Garet--Marchand~\cite{GM-contact}. 

Note that the one-dimensional case is simpler because the growth of the supercritical contact process in dimension 1 is characterized by the right-edge velocity: its continuity is proved in Liggett~\cite{MR776231}, Theorem~3.36. See also Durrett \cite{MR757768} for an analogous result about 2D oriented percolation.

In Section~2, we introduce the notation,  build contact processes with distinct infection parameters on the same space thanks to the Harris construction and recall the definition and properties of the essential hitting time introduced in \cite{GM-contact}. Section~3 is devoted to the proof of the left-continuity, while in Section~4 we prove the right-continuity.


\section{Notation and known results}

%
%
We work on the grid $\Zd$, with $d\ge 1$, and we put an edge between any pair of sites at distance $1$ for $\|.\|_1$. We denote by $\Ed$ the set of these edges.

To define the contact process, we use the Harris construction~\cite{MR0488377}. It allows to couple contact processes starting from distinct initial configurations and distinct parameters $\lambda \in(0, \lambda_{\max}]$, where $\lambda_{\max}>0$ is fixed and finite, by building them from a single collection of Poisson measures on~$\R_+$.

\subsection{Construction of the Poisson measures }

We endow $\R_+$ with the Borel $\sigma$-algebra $\mathcal B(\R_+)$, and we denote by $M$ the set of locally finite counting measures $m=\sum_{i=0}^{+\infty} \delta_{t_i}$. We endow this set with the $\sigma$-algebra $\mathcal M$ generated by the maps $m\mapsto m(B)$, where $B$ describes the set of Borel sets in $\R_+$. 

As the continuity is a local property, it will be sufficient in the sequel to build a coupling for contact processes with parameters in $(0,\lambda_{\max}]$, for a fixed and well chosen $\lambda_{\max}>0$. We define the measurable space $(\Omega, \mathcal F)$ by setting
$$\Omega=M^{\Ed}\times M^{\Zd}\times ([0,\lambda_{\max}]^{\N})^{\Ed} \text{ and } \mathcal F=\mathcal{M}^{\otimes \Ed} \otimes \mathcal{M}^{\otimes \Zd}\otimes  ([0,\lambda_{\max}]^{\otimes \N})^{\otimes\Ed}.$$
On this space, we consider the  probability measure defined by 
$$\P=\mathcal{P}_{\lambda_{\max}}^{\otimes\Ed} \otimes \mathcal{P}_1^{\otimes\Zd} \otimes  (U([0,\lambda_{\max}])^{\otimes \N})^{\otimes\Ed},$$
where, for every $\lambda\in\R_+$, $\mathcal{P}_{\lambda}$ is the law of a Poisson point process on $\R_+$ with intensity $\lambda$ and $U([a,b])$ is the uniform law on the compact set $[a,b]$. 

Fix an edge $e$. Denoting by $(S_i^e)_{i\ge 1}$ the atoms of $\omega_e$, we build the classical coupling between the Poisson measures of the infection processes with different parameters $\lambda \in (0, \lambda_{\max}]$. Define 
$$m_\lambda^e=m_{\lambda}(\omega_e,(U_i^e)_{i \ge 1}) =\sum_{i=1}^{+\infty}\1_{\{U_i^e\le\frac{\lambda}{\lambda_{\max}}\}} \delta_{S_i^e}.$$
Under $\P$, the random variable $m_{\lambda}$ is a Poisson point process with parameter $\lambda$.
We then define, for $\lambda \le \lambda_{\max}$, the application
$$\begin{array}{rcl}
\Psi_\lambda: \Omega & \longrightarrow & M^{\Ed}\times M^{\Zd} \\
((\omega_e)_{e \in \Ed}, (\omega_z)_{z \in \Zd},(U^i_e)_{e\in\Ed,i\ge 1})  & \longmapsto & ((m_{\lambda}(\omega_e,(U_i^e)_{i \ge 1}))_{e \in \Ed}, (\omega_z)_{z \in \Zd}).
\end{array}$$
The law of $\Psi_\lambda$ under $\P$ is then
$$\P_\lambda= \mathcal{P}_{\lambda}^{\otimes\Ed} \otimes \mathcal{P}_1^{\otimes\Zd}.$$
We thus recover infection processes, indexed by $\Ed$, with parameter $\lambda$ and recovering processes, indexed by $\Zd$, with parameter 1.
Note that the Poisson measures for recoverings, $(\omega_z)_{z \in \Zd}$, do not depend on $\lambda$. The following lemma will be useful to compare the evolution of two contact processes with different parameters.

\begin{lemme}
\label{continuite-des-boites}
Let $t>0$ and let $S$ be a finite subset of $\Ed$.
Assume $0<\lambda'\le\lambda \le \lambda_{\max}$ and note
$$\Idem(S,t,\lambda,\lambda')=\miniop{}{\cap}{e\in S}\left\{\restriction{m^e_{\lambda}}{[0,t]}= \restriction{m^e_{\lambda'}}{[0,t]}\right\}.$$
For each $\epsilon>0$, there exists $\delta=\delta(S,t,\epsilon)>0$ such that
$$\forall \lambda,\lambda' \in (0, \lambda_{\max}] \quad | \lambda'-\lambda| \le \delta \; \Rightarrow \; \P(\Idem(S,t,\lambda,\lambda'))\ge 1-\epsilon.$$
\end{lemme}

\begin{proof} Let $\lambda,\lambda' \in (0, \lambda_{\max}]$, and assume without loss of generality that $\lambda \le \lambda'$.

For each $e\in\Ed$ and $t>0$, set
\begin{align*}
D^e_t & =\sum_{i=1}^{+\infty}\1_{\{\frac{\lambda'}{\lambda_{\max}}<U_i^e\le\frac{\lambda}{\lambda_{\max}}\}}\1_{\{S_i^e\le t\}}, \\
\text{then } \E (D^e_t) & = \frac{\lambda -\lambda'}{\lambda_{\max}}\E(\omega_e([0,t]))=\frac{\lambda -\lambda'}{\lambda_{\max}}\lambda_{\max}t= (\lambda'-\lambda)t.
\end{align*}
Now,  
\begin{align*}
\P(\Idem(S,t,\lambda,\lambda')^c) & \le \sum_{e \in S} \P(\restriction{m^e_{\lambda}}{[0,t]} \ne \restriction{m^e_{\lambda'}}{[0,t]}) \\
&\le \sum_{e \in S}\P(D^e_t\ge 1) \le \sum_{e \in S}\E (D^e_t) \le |S|t (\lambda'-\lambda),
\end{align*}
so we can take $\delta=1/({t|S|\epsilon})$.
\end{proof}
 
\subsection{Graphical construction of the contact process} 
\label{SubS:constr}

This construction is exposed in all details in Harris~\cite{MR0488377}; we just give here an informal description. 
Suppose that $\lambda\in (0,\lambda_{\max}]$ is fixed.
Let $\omega=((\omega_e)_{e \in \Ed}, (\omega_z)_{z \in \Zd},(U^i_e)_{e\in\Ed,i\ge 1}) \in \Omega$. Above each site $z \in \Zd$, we draw a time line $\R_+$, and we put a cross at the times given by $\omega_z$, corresponding to  potential recoverings at site $z$. Above each edge $e \in \Ed$, we draw at the times given by $m_{\lambda}((\omega_e)_{e \in \Ed},(U^i_e)_{e\in\Ed,i\ge 1})$ an horizontal segment between the extremities of the edge, corresponding to a potential infection through edge $e$ (remember we fix the infection rate $\lambda$). 

An open path follows the time lines above sites -- but crossing a cross is forbidden -- and uses horizontal segments to jump from a time line to a neighboring time line: in this description, the evolution of the contact process looks like a percolation process, oriented in time but not in space. 
For $x,y \in \Zd$ and $t \ge 0$, we say that $\xi_t^{\lambda,x}(y)=1$ if and only if there exists an open path from $(x,0)$ to $(y,t)$, then we define:
\begin{eqnarray}
\xi_t^{\lambda,x} & = & \{y \in \Zd: \; \xi_t^{\lambda,x}(y)=1\}, \nonumber \\
\forall A \in \mathcal P(\Zd) \quad  \xi_t^{\lambda,A} & = & \bigcup_{x \in A} \xi_t^{\lambda,x}. \label{additivite}
\end{eqnarray}
For instance, we obtain
\begin{eqnarray*}
(A \subset B, \quad \lambda \le \lambda') &  \Rightarrow  & (\forall t \ge 0\quad \xi_t^{\lambda,A} \subset \xi_t^{\lambda',B}).
\end{eqnarray*}

Harris proved that under $\P$, or under $\P_\lambda$, the process $(\xi^{\lambda,A}_t)_{t \ge 0}$ is the contact process with infection rate $\lambda$, starting from initial configuration $A$. 

\subsection{Translations}

For $t \ge 0$, we define the translation operator $\theta_t$ on a locally finite counting measure $m=\sum_{i=1}^{+\infty} \delta_{t_i}$ on $\R_+$ by setting
$$\theta_t m=\sum_{i=1}^{+\infty} \1_{\{t_i\ge t\}}\delta_{t_i-t}.$$
The translation $\theta_t$ induces an operator on $\Omega$, still denoted by $\theta_t$: \\
for every $\omega=((\omega_e)_{e \in \Ed}, (\omega_z)_{z \in \Zd},(U^i_e)_{e\in\Ed,i\ge 1}) \in \Omega$, we set
$$ \theta_t(\omega)=((\theta_t \omega_e)_{e \in \Ed}, (\theta_t \omega_z)_{z \in \Zd},(U^{i+\omega_e([0,t])}_e)_{e\in\Ed,i\ge 1}  ).$$
Since the Poisson point processes are translation invariant and $\omega_e([0,t])$ is independent from the $(U^i_e)$'s, $\P$ and  $\P_\lambda$ are  invariant under~$\theta_t$.

There is also a natural action of $\Zd$ on $\Omega$, which preserves $\P$ and  $\P_\lambda$, and which consists in changing  the observer's point of view:
for   $x \in \Zd$, we define the translation operator~$T_x$ by setting: 
$$\forall \omega \in \Omega\quad T_x (\omega)=( (\omega_{x+e})_{e \in \Ed}, ( \omega_{x+z})_{z \in \Zd},(U^i_{x+e})_{e\in\Ed;i\ge 1}),$$
where $x+e$ the edge $e$ translated by vector $x$.

\subsection{Notation and classical estimates for the contact process}

For a set $A \subset \Zd$, we define the life time $\tau^A_{\lambda}$ of the process starting from $A$ by
$$\tau^A_{\lambda}=\inf\{t\ge0: \; \xi_t^{\lambda,A}=\varnothing\}. $$
If $y\in\Zd$, we write  $\tau^y_{\lambda}$ instead of  $\tau^{\{y\}}_{\lambda}$ and we simply write $\tau_{\lambda}$ for $\tau^0_{\lambda}$. We also note 
$$\{\tau_\lambda=+\infty\}=\{0 \communiquelambda \infty\}.$$
The critical parameter for the contact process in $\Zd$ is then
$$\lambda_c(\Zd) =\inf\{\lambda>0: \; \P_\lambda(\tau_{\lambda}=+\infty)>0\} \in (0,+\infty).$$
Define, for $\lambda>\lambda_c(\Zd)$, the following conditional probability
$$ \Pbarre_\lambda(.)=\P_\lambda(.|\tau_{\lambda}=+\infty)=\frac{\P(\; . \; \cap \{0 \communiquelambda \infty\})}{\P(0 \communiquelambda \infty)}.$$

For $A \subset \Zd$ and $x \in \Zd$,  we also define the first infection time $t^A_{\lambda}(x)$ of site $x$ from set $A$ by 
$$t^A_{\lambda}(x)=\inf\{t\ge 0: \; x \in \xi_t^{\lambda,A}\}.$$
If $y\in\Zd$, we write  $t^y_{\lambda}(x)$ instead of  $t^{\{y\}}_{\lambda}(x)$ and we simply write $t_{\lambda}(x)$ for $t^0_{\lambda}(x)$. The set of points infected before time $t$ is then 
\begin{align*}
H^{\lambda}_t & =  \{x\in\Zd: \; t_{\lambda}(x)\le t\}\quad \text{and}\; \tilde{H}^{\lambda}_t=H_t+[0,1]^d.
\end{align*}

The following estimates are classical for the supercritical contact process; they are stated here with some extra uniformity in the parameter $\lambda$ (this uniformity is mainly obtained by stochastic comparison):
\begin{prop}[Proposition 5 in Garet--Marchand~\cite{GM-contact}]
\label{propuniforme}\label{propmoments}
Let $\lambda_{\min},\lambda_{\max}$ with $\lambda_c(\Zd)<\lambda_{\min}\le \lambda_{\max}$. 
There exist $A,B,C,c,\rho>0$ such that for every
$\lambda\in [\lambda_{\min},\lambda_{\max}]$, for every  $x \in \Zd$, for every  $ t\ge0$, 
\begin{eqnarray}
\P(\tau_{\lambda}=+\infty) & \ge & \rho,
\label{uniftau} \\
\P(H^{\lambda}_t \not\subset [-Ct,Ct]^d ) & \le & A\exp(-Bt), 
\label{richard} \\
\P ( t<\tau_{\lambda}<+\infty) &\le&  A\exp(-Bt),\label{grosamasfinis} \\
 \P\left( t_{\lambda}(x)\ge \frac{\|x\|}c+t,\; \tau_{\lambda}=+\infty \right) & \le & A\exp(-Bt).
\label{retouche}
\end{eqnarray}
\end{prop}

\subsection{Essential hitting times and shape theorem}
We now recall the definition of the essential hitting time $\sigma_{\lambda}(x)$. It was introduced in \cite{GM-contact} to prove an asymptotic shape result for the supercritical contact process in random environment. See also Garet--Marchand~\cite{GM-contact-gd} and Garet--Gou\'er\'e--Marchand~\cite{GGM} for further uses. The essential hitting time $\sigma_{\lambda}(x)$ is a time when the site $x$ is infected from the origin $0$ and also has an infinite life time. It is defined through a family of stopping times as follows: we set $u_0(x)=v_0(x)=0$ and we define recursively two increasing sequences of stopping times $(u_n(x))_{n \ge 0}$ and $(v_n(x))_{n \ge 0}$ with
$u_0(x)=v_0(x)\le u_1(x)\le v_1(x)\le u_2(x)\dots$:
\begin{itemize}
\item Assume that $v_k(x)$ is defined. We set $u_{k+1}(x)  =\inf\{t\ge v_k(x): \; x \in \xi^{\lambda,0}_t \}$. 
\item Assume that $u_k(x)$ is defined, with $k \ge 1$. We set $v_k(x)=u_k(x)+\tau^x_{\lambda}\circ \theta_{u_k(x)}$.
\end{itemize}
We then set
\begin{equation}
\label{definitiondeK}
K_{\lambda}(x)=\min\{n\ge 0: \; v_{n}(x)=+\infty \text{ or } u_{n+1}(x)=+\infty\}.
\end{equation}
This quantity represents the number of steps before we stop: either we stop because we have just found an infinite $v_n(x)$, which corresponds to a time $u_n(x)$ when $x$ is occupied and has infinite progeny, or we stop because we have just found an infinite $u_{n+1}(x)$, which says that after $v_n(x)$, site $x$ is never infected anymore. When the contact process survives, the second case does almost surely not occur.

In \cite{GM-contact}, using \eqref{grosamasfinis} and \eqref{retouche}, it is proved that $K_{\lambda}(x)$ is almost surely finite, which allows to define the essential hitting time $\sigma_{\lambda}(x)$ by setting $$\sigma_{\lambda}(x)=u_{K_{\lambda}(x)}.$$  It is of course larger than the hitting time $t_{\lambda}(x)$ and can been seen as a regeneration time.
At the same time, we define the operator $\tilde \Theta_x^{\lambda}$ on $\Omega$ by:
\begin{equation*}
\tilde \Theta_x^{\lambda} = 
\begin{cases} T_{x} \circ \theta_{\sigma_{\lambda}(x)} & \text{if $\sigma_{\lambda}(x)<+\infty$,}
\\
T_x &\text{otherwise.}
\end{cases}
\end{equation*}
The advantage of the essential hitting time $\sigma_{\lambda}(x)$ is that it enjoys, unlike $t_{\lambda}(x)$, 
some good invariance and integrability properties when conditioned to survive.  
We now recall the main results of \cite{GM-contact} we will need here. In the following, we fix $\lambda_{\min},\lambda_{\max}>0$ such that $\lambda_c(\Zd)<\lambda_{\min}\le \lambda_{\max}$.

\begin{prop}[Garet--Marchand \cite{GM-contact}, Theorems 1 and 3, Corollary 21, Theorem~22 and Lemma 29] $\;$ 
\begin{itemize}

\item
For each $\lambda>\lambda_c(\Zd)$, for every $x \in \Zd$, 
\begin{equation}  \label{invariance} 
\text{the probability measure $\Pbarre_\lambda$ is invariant under the map $\tilde \Theta_x^{\lambda}.$}
\end{equation}

\item There exist constants $(C_p)_{p\ge 1}$ such that for every
$\lambda\in [\lambda_{\min},\lambda_{\max}]$, for every  $x \in \Zd$, for every  $p\ge 1$,
\begin{equation} \label{asigma}
\Ebarre_{\lambda} [\sigma_{\lambda}(x)^p ] \le  C_p (1+\|x\|)^{p}. 
\end{equation}

\item
For each $\lambda>\lambda_c(\Zd)$, for every $x \in \Zd$, there exists a deterministic $\mu_{\lambda}(x)$ such that
\begin{equation} \label{basique}
\lim_{n \to +\infty} \frac{t_{\lambda}(nx)}n = \lim_{n \to +\infty}\frac{\sigma_{\lambda}(nx)}n=\mu_{\lambda}(x).
\end{equation}
The convergence holds $\Pbarre_{\lambda}$ almost surely, and also in $L^1(\Pbarre_{\lambda})$. 

\item 
The function $x\mapsto \mu_{\lambda}(x)$ can be extended to a norm on $\Rd$. Let 
\begin{equation}
\shape{\lambda}=\{x \in \Rd: \; \mu_\lambda(x)\le 1\}.
\end{equation}

\item 
For every $\epsilon>0$, $\Pbarre_{\lambda}-a.s.$, for every~$t$ large enough,
\begin{equation}
\label{leqdeforme}
(1-\epsilon)\shape{\lambda}\subset \frac{\tilde H^{\lambda}_t}t\subset (1+\epsilon)\shape{\lambda}.
\end{equation}
\end{itemize}
\end{prop}

It has also been noted in \cite{GM-contact} (see proof of Lemma 25 there) that there exists $M_1>0$ such that, for each $\lambda\in [\lambda_{\min},\lambda_{\max}]$ and each $x\in\Zd\backslash\{0\}$, the sequence $(\Ebarre_{\lambda}\sigma_{\lambda}(nx)+M_1)_{n\ge 1}$ is subadditive. Thus, with~\eqref{basique}, we can represent  $\mu_{\lambda}(x)$ as the following infimum:
\begin{equation}
\label{cestuninf}
\forall \lambda\in [\lambda_{\min},\lambda_{\max}] \quad \forall x \in \Zd \quad \mu_{\lambda}(x)=\miniop{}{\inf}{n\ge 1}\frac{M_1+\Ebarre_{\lambda}(\sigma_{\lambda}(nx))}n.
\end{equation}

As a corollary of \eqref{basique}, we obtain the following monotonicity property:

\begin{coro}
\label{coromonotonie}
For each $x\in\Zd$, $\lambda\mapsto \mu_{\lambda}(x)$ is non-increasing on $(\lambda_c(\Zd),+\infty)$.
\end{coro}

\begin{proof}
Suppose $\lambda_c(\Zd)<\lambda'<\lambda<+\infty$. Choose $\lambda_{\min},\lambda_{\max}$ with $\lambda_c(\Zd)<\lambda_{\min}<\lambda'<\lambda \le \lambda_{\max}$. Use the construction of Subsection \ref{SubS:constr} to build the two contact processes with respective parameters $\lambda$ and $\lambda'$. On the event $\{0\communiquelambdap\infty\}$, which has positive probability, we have that for each $n\ge 1$, $\frac{t_{\lambda}(nx)}{n}\le \frac{t_{\lambda'}(nx)}{n}$.
Letting $n$ go to infinity, we get $\mu_{\lambda'}(x)\le \mu_{\lambda}(x)$.
\end{proof}

\section{Left-Continuity}

We prove here the left-continuity of $\mu_\lambda$. More precisely, we prove that for each $\lambda_0>\lambda_c(\Zd)$, for every $\varepsilon>0$, there exists $\delta>0$ such that
\begin{equation}
\label{E:contgauche}
\forall \lambda \in [\lambda_0- \delta, \lambda_0] \quad \forall x\in \mathbb S^{d-1}\quad  |\mu_{\lambda_0}(x)-\mu_{\lambda}(x)|\le \varepsilon.
\end{equation}
 
\medskip
When proving continuity theorems for the time constant in first passage percolation (see Cox and Kesten \cite{MR588407,MR633228,kesten}),
the left-continuity is usually considered as the easy part, due to the fact that the time constant is an infimum. 
In the case of the contact process, there are extra difficulties,
because contact processes with different intensities can not be coupled in such a way that they die simultaneously.

\begin{lemme}
\label{unemonotonie}
Let $\lambda>\lambda_c(\Zd)$. For each $x \in \Zd$, 
$\displaystyle \miniop{}{\limsup}{\lambda'\to\lambda^{-}}\Ebarre_{\lambda'}(\sigma_{\lambda'}(x))\le \Ebarre_{\lambda}(\sigma_{\lambda}(x)).$
\end{lemme}

\begin{proof}
Fix $\lambda>\lambda_c(\Zd)$. Choose $\lambda_{\min}$ such that $\lambda_c(\Zd)<\lambda_{\min}< \lambda$ and set $\lambda_{\max}=\lambda$. Fix $x \in \Zd$. Use the construction of Subsection \ref{SubS:constr}.
  
In this proof, for $\lambda' \in [\lambda_{\min}, \lambda]$, we note $\sigma_{\lambda'}=\sigma_{\lambda'}(x)$.
Suppose that for every $\lambda' \in [\lambda_{\min}, \lambda]$, we have managed to construct a ``good'' event $G(\lambda')$  such that $\sigma_{\lambda'}=\sigma_{\lambda}$ on $G(\lambda')$. Then, for $\lambda' \in [\lambda_{\min}, \lambda]$,
\begin{align*}
 {\E} \left( \sigma_{\lambda'}, \; 0\communiquelambdap\infty \right)
& = {\E} \left( \sigma_{\lambda'}, \; 0\communiquelambdap\infty, \;G(\lambda') \right) + {\E} \left( \sigma_{\lambda'}, \; 0\communiquelambdap\infty, \; G(\lambda')^c\right)\\
& \le {\E} \left( \sigma_{\lambda}, \; 0\communiquelambdap\infty \right) +  \sqrt{\E \left( \sigma_{\lambda'}^2, \; 0\communiquelambdap\infty \right)} \sqrt{\P ( 0\communiquelambdap\infty,\; G(\lambda' )^c)}.
\end{align*}
Now, using the fact that   $\{0\communiquelambdap +\infty\} \subset  \{0\communiquelambda +\infty\}$ and the control \eqref{asigma} on the moments of $\sigma_{\lambda'}$, we get
\begin{align}
\Ebarre_{\lambda'}(\sigma_{\lambda'}) 
& =  \frac{{\E} \left( \sigma_{\lambda'}, \; 0\communiquelambdap\infty \right) }{\P(0\communiquelambdap\infty)}   \nonumber \\
& \le \frac{\P (0\communiquelambda +\infty)}{\P(0\communiquelambdap +\infty)}\Ebarre_{\lambda}(\sigma_{\lambda}) + \sqrt{\Ebarre_{\lambda'} \left( \sigma_{\lambda'}^2 \right)} \sqrt{\frac{\P (0\communiquelambda +\infty)}{\P(0\communiquelambdap +\infty)}}
\sqrt{ \Pbarre_\lambda (  G(\lambda' )^c)} \nonumber \\
& \le\frac{\P (0\communiquelambda +\infty)} {\P(0\communiquelambdap +\infty)}\left(\Ebarre_{\lambda}(\sigma_{\lambda})+\sqrt{C_2(1+\|x\|^2)}\sqrt{\Pbarre_{\lambda}(G(\lambda')^c)}\right).
\label{E:ebarresigmaprime}
\end{align}
Thus if we prove that  $\frac{\P (0\communiquelambda +\infty)} {\P(0\communiquelambdap +\infty)}$ 
and $\Pbarre_{\lambda}(G(\lambda'))$ are close to $1$ when $\lambda'$ is close to $\lambda$, we can complete the proof. 

We now build the ``good'' event $G(\lambda')$  such that $\sigma_{\lambda'}=\sigma_{\lambda}$ on $G(\lambda')$ and such that $\Pbarre_{\lambda}(G(\lambda'))$ goes to $1$ as $\lambda'$ goes to $\lambda$. 
Since $\sigma_{\lambda}$ is $\Pbarre_{\lambda}$-a.s. finite and $H^{\lambda}_{\sigma_{\lambda}}$ is $\Pbarre_{\lambda}$-a.s. a finite set, we can first choose $M>0$ such that 
\begin{equation}
\label{E:choixM}
\Pbarre_{\lambda}(A_M)\ge 1-\frac{\epsilon}{3},\text{ where } A_M=\{H^{\lambda}_{\sigma_{\lambda}} \subset [-M,M]^d\} \cap \{\sigma_{\lambda}\le M\}.
\end{equation}
Then, estimates \eqref{richard} and \eqref{grosamasfinis} let us choose $L>0$ such that for each $\lambda'\in [\lambda_{\min},\lambda]$
\begin{equation}
\label{E:choixL}
\Pbarre_{\lambda}(B_L(\lambda'))\ge 1-\frac{\epsilon}{3},\text{ where }B_L(\lambda')=\{H^{\lambda'}_L\subset [-CL,CL]^d\}\cap \{L<\tau_{\lambda'}<\infty\}^c.
\end{equation}
Set $S=[-(M+CL),(M+CL)]^d\cap \Zd$ and $t=M+L$.
With Lemma~\ref{continuite-des-boites}, we can choose $\delta>0$ such that
\begin{equation}
\label{E:choixdelta}
\forall \lambda'\in [\lambda-\delta,\lambda] \quad \Pbarre_{\lambda}(\Idem(S,t,\lambda,\lambda'))\ge 1-\epsilon/3.
\end{equation}
Finally, we consider, for every $\lambda'\in [\lambda-\delta,\lambda]$, the event
\begin{align*}
G(\lambda')& =A_M \cap \tilde{\Theta}_{x,\lambda}^{-1}(B_L(\lambda'))
 \cap \Idem(S,t,\lambda,\lambda').
\end{align*}
The choices \eqref{E:choixM}, \eqref{E:choixL} and \eqref{E:choixdelta} we respectively made for $M,L$ and $\delta$, and the invariance property \eqref{invariance} ensure that 
$$\Pbarre_{\lambda}(G(\lambda'))\ge 1-\epsilon.$$
The event $A_M$ says that, apart from the fact that $(x,\sigma_{\lambda})\communiquelambda\infty$,  the time $\sigma_{\lambda}$ is determinated by the configuration of the Poisson processes in the space-time box $[-M,M]^d\times [0,M]$.
The event $\tilde{\Theta}_{x,\lambda}^{-1}(B_L(\lambda'))$ says that, if $(x,\sigma_{\lambda})$ has a progeny for parameter $\lambda'$ still alive at time  $\sigma_{\lambda}+L$, then $(x,\sigma_{\lambda})\communiquelambda\infty$.
The event $\Idem(S,t,\lambda,\lambda')$ says that the infection at rate $\lambda'$ in the box $S\times [0,t]$ behaves exactly like the infection at rate $\lambda$ in the same box. 

Now, on the event $G(\lambda')\cap\{0\communiquelambda \infty\}$,  the point $(x,\sigma_{\lambda})$ has a progeny for parameter $\lambda$ that is still alive at time $\sigma_{\lambda}+L$. But the infections with rate $\lambda$ and $\lambda'$ coincide in the box $S\times [0,t]$, so the point $(x,\sigma_{\lambda})$ has a progeny for parameter $\lambda'$ that is also still alive at times $\sigma_{\lambda}+L$. Then $(x,\sigma_{\lambda})\communiquelambdap\infty$, and it is now easy to see that $\sigma_{\lambda'}=\sigma_{\lambda}$. Note also that
$G(\lambda')\cap \{\tau_{\lambda}=+\infty\} \subset \{ \tau_{\lambda'}=+\infty\} $.
 This gives
\begin{align*}
\frac{\P (0\communiquelambda +\infty)} {\P(0\communiquelambdap +\infty)}&=\frac1{\Pbarre_{\lambda}(0\communiquelambdap +\infty)}\le \frac1{\Pbarre_{\lambda}(G(\lambda'))}\le\frac1{1-\epsilon}
\end{align*}
and, coming back to \eqref{E:ebarresigmaprime}, we see that 
\begin{align*}
\forall \lambda' \in [\lambda-\delta,\lambda] \quad \Ebarre_{\lambda'}(\sigma_{\lambda'}) & \le \frac1{1-\epsilon} \left( \Ebarre_{\lambda}(\sigma_{\lambda})+\sqrt{C_2(1+\|x\|^2)} \sqrt{\epsilon}\right).
\end{align*}
This completes the proof.
\end{proof}

\begin{lemme}
\label{lemmgauche}
For each $x\in\Zd$, $\lambda\mapsto \mu_{\lambda}(x)$ is left-continuous on $(\lambda_c(\Zd),+\infty)$.
\end{lemme}

\begin{proof}
Fix $x \in \Zd$. Since, from Corollary \ref{coromonotonie}, the application $\lambda\mapsto \mu_{\lambda}(x)$ is non-increasing on $(\lambda_c(\Zd),+\infty)$, we can define
 $$L=\miniop{}{\lim}{\lambda'\to \lambda^-}\mu_{\lambda'}(x).$$
  Obviously $L\ge \mu_{\lambda}(x)$ and we must prove $L\le \mu_{\lambda}(x)$. Note $\lambda_n=\lambda-1/n$.
Using the representation \eqref{cestuninf} of $\mu_\lambda(x)$ as an infimum, we have
\begin{align*} 
L & = \miniop{}{\inf}{n\ge 1} \mu_{\lambda_n}(x) 
= \miniop{}{\inf}{n\ge 1}\miniop{}{\inf}{k\ge 1} \frac{\Ebarre_{\lambda_n} (\sigma_{\lambda_n}(kx)) +M_1}k\\
&=\miniop{}{\inf}{k\ge 1}\miniop{}{\inf}{n\ge 1}\frac{\Ebarre_{\lambda_n}(\sigma_{\lambda_n}(kx))+M_1}k\\
&=\miniop{}{\inf}{k\ge 1} \left( \frac{M_1}{k} +\miniop{}{\inf}{n\ge 1} 
\frac{\Ebarre_{\lambda_n}(\sigma_{\lambda_n}(kx))}{k} \right)\\
\end{align*}
By Lemma~\ref{unemonotonie}, for each $k$, $\miniop{}{\inf}{n\ge 1}\Ebarre_{\lambda_n}(\sigma_{\lambda_n}(kx)) \le \Ebarre_{\lambda}(\sigma_{\lambda}(kx))$, so
\begin{align*} 
L&\le \miniop{}{\inf}{k\ge 1}\left(\frac{M_1}{k}+\frac{\Ebarre_{\lambda}(\sigma_{\lambda}(kx))}{k}\right)=\mu_{\lambda}(x),
\end{align*}
 which completes the proof.
\end{proof}

The difference between (\ref{E:contgauche}) and Lemma \ref{lemmgauche} is the uniformity of the control. For all $\lambda >0$, since $\mu_{\lambda}$ is a norm and by symmetry of the model, we have for all $x, y \in \Rd$,
$$ \mu_\lambda (x) - \mu_\lambda (y)  \leq \mu_\lambda (x-y) \leq  \|x-y\|_1 \mu_\lambda (e_1) \,,$$
where $e_1 = (1,0,\dots, 0)$. We obtain that $ | \mu_\lambda (x) - \mu_\lambda (y) |  \leq \|x-y\|_1 \mu_\lambda (e_1) $. Fix $\lambda_0\in (\lambda_c, +\infty)$ and $\varepsilon >0$. By Lemma \ref{lemmgauche} we know that $\lim_{\lambda \rightarrow \lambda_0^-} \mu_\lambda (e_1) = \mu_{\lambda_0} (e_1) $, thus there exists $\delta >0$ such that for all $\lambda \in [\lambda_0 - \delta , \lambda_0]$, for all $x,y\in \Rd$, we have $| \mu_\lambda (x) - \mu_\lambda (y) |  \leq 2 \|x-y\|_1 \mu_{\lambda_0 }(e_1) $. We obtain the existence of $\eta >0$ such that for all $x,y\in \Rd$ satisfying $\| x-y \|_1 \leq \eta$, we have
$$ \sup_{\lambda \in [\lambda_0 - \delta , \lambda_0]} \{ | \mu_\lambda (x) - \mu_\lambda (y) |  \}  \leq \varepsilon \,.$$
There exists a finite set of points $y_1,\dots , y_m$ in $\Rd$ such that
$$ \mathbb{S}^{d-1} \subset \bigcup_{i=1}^m \{ x\in \Rd \,:\, \|x-y_i\|_1 \leq \eta \} \,,$$
thus for all $\lambda \in [\lambda_0 - \delta , \lambda_0]$ we obtain
$$ \sup_{x\in \mathbb{S}^{d-1}} |\mu_\lambda (x) - \mu_{\lambda_0} (x)| \leq 2 \eps + \max_{i=1,\dots , m}|\mu_\lambda (y_i) - \mu_{\lambda_0} (y_i)| \,.$$
By homogeneity of $\mu_{\lambda}$, the result of Lemma \ref{lemmgauche} also holds for all $y\in \Rd$, in particular for $y_i, i\in \{1,\dots, m\}$. This concludes the proof of (\ref{E:contgauche}).

We can notice that the previous argument also applies to the study of the right-continuity of $\mu_\lambda$. However, as we will see in the next section, we do not need it since we perform directly the study of the right-continuity of $\mu_\lambda$ uniformly in all directions.

\section{Right-continuity}

We prove here the right-continuity of $\mu_\lambda$. More precisely, we prove that for each $\lambda_0>\lambda_c(\Zd)$, for every $\varepsilon>0$, there exists $\delta>0$ such that
\begin{equation}
\label{E:contdroite}
\forall \lambda \in [\lambda_0, \lambda_0+\delta] \quad \forall x\in\mathbb{S}^{d-1}\quad  |\mu_{\lambda_0}(x)-\mu_{\lambda}(x)|\le \varepsilon.
\end{equation}

As we will see, the right-continuity of the asymptotic shape of the contact process can be obtained by a slight modification of a part of the proof of the large deviations inequality for the contact process established by Garet and Marchand in~\cite{GM-contact-gd}.  

Let $\lambda_0>\lambda_c(\Zd)$ be fixed. Fix $\lambda_{\min},\lambda_{\max}$ with $\lambda_c(\Zd)<\lambda_{\min}\le \lambda_0<\lambda_{\max}$.

Let $\alpha,\epsilon>0$ and $L,N$ be positive integers.
Consider $\lambda \ge \lambda_0$ and close to $\lambda_0$. 
We define the following event, relative to the space-time box $B_N=B_N(0,0)=([-N,N]^d\cap \Zd)\times[0,2N]$:
\begin{eqnarray*}
A_{\lambda,\lambda_0}^{\alpha,L,N,\epsilon}= \left\{\forall (x_0,t_0) \in B_N\quad 
  \xi^{x_0,\lambda}_{\alpha L N-t_0}\circ \theta_{t_0} \subset x_0+(1+\epsilon)(\alpha L N-t_0)\shape{\lambda_0}\right\}\\
 \cap \left\{\forall (x_0,t_0) \in B_N
 \miniop{}{\cup}{0\le s\le\alpha L N-t_0}  \xi^{x_0,\lambda}_{s}\circ \theta_{t_0} \subset ]-LN,LN[^d\right\}.
\end{eqnarray*}
Consider first $A_{\lambda_0,\lambda_0}^{\alpha,L,N,\epsilon}$.
The first part of the event ensures that the descendants, at time $\alpha L N$, of any point $(x_0,t_0)$ in the box $B_N$ are included in $x_0+(1+\epsilon)(\alpha L N)\shape{\lambda}$: it is a sharp control, requiring the asymptotic shape Theorem for parameter $\lambda_0$. The second part ensures that the descendants, at all times in $[0,\alpha L N] $, of the whole box $B_N$ are included in $]-LN,LN[^d$: the bound is rough, only based on the (at most) linear growth of the process with parameter $\lambda_0$. Thus, the "good growth" event $A_{\lambda_0,\lambda_0}^{\alpha,L,N,\epsilon}$ is typical, and it has been proved that

\begin{lemme}[\cite{GM-contact-gd}]
\label{catendvers12}
Fix $\lambda_0>\lambda_c(\Zd)$. 
There exists $\alpha=\alpha(\lambda_0)\in (0,1)$ such that for every $\epsilon \in (0,1)$, every $L_0>0$, there exists an integer $L>L_0$ such that
$$\lim_{N \to +\infty} \P(A_{\lambda_0,\lambda_0}^{\alpha,L,N,\epsilon})=1.$$
\end{lemme}

Garet and Marchand used Lemma~\ref{catendvers12}
to prove the upper large deviations for the contact process: for every $\lambda_0\in [\lambda_{\min},\lambda_{\max}]$, provided that
 $\alpha=\alpha(\lambda_0)$ is fixed as in Lemma~\ref {catendvers12}, then for $L$ greater than some $L_0=L>L_0(\epsilon,\lambda_0)$, they prove that there exists 
$p_1=p_1(\lambda_0,\epsilon,L)>0$ such that
$$\P(A_{\lambda_0,\lambda_0}^{\alpha,L,N,\epsilon/3})>p_1\Longrightarrow \exists A,B\quad \forall t>0\quad \P(\xi^{0,\lambda_0}_t \not \subset (1+\epsilon)t\shape{\lambda_0})\le A\exp(-B t).$$
The idea of the proof is as follows: a too fast infection from $(0,0)$ to $\Zd \times\{n\}$ uses a too fast path, along which we should find a number of order $\theta n$ of "bad growth" events, \ie translated versions of $(A_{\lambda_0,\lambda_0}^{\alpha,L,N,\epsilon/3})^c$. The proof ends with a Peierls argument: the event $(A_{\lambda_0,\lambda_0}^{\alpha,L,N,\epsilon/3})^c$ is local, thus its translated events are only locally dependent. If their probability is small enough, we can control the probability that there exists a path from $(0,0)$ to $\Zd \times\{n\}$ with at least $\theta n$ "bad growth" events.

Let's come back to the right-continuity. 
Fix $\lambda_0>\lambda_c(\Zd)$ and $\epsilon>0$. Take $\alpha$ given by Lemma~\ref{catendvers12}, $L$ large enough, and $p_1(\lambda_0,\epsilon,L)>0$ as before.
A look at the proof of the Peierl argument in \cite{GM-contact-gd} should convince the reader that for each $\lambda\ge \lambda_0$, we have 
 $$\P(A_{\lambda,\lambda_0}^{\alpha,L,N,\epsilon/3})>p_1\Longrightarrow \exists A,B\quad  \forall t>0\quad \P(\xi^{0,\lambda}_t \not \subset (1+\epsilon)t\shape{\lambda_0})\le A\exp(-B t).$$
Remember that the event $A_{\lambda,\lambda_0}^{\alpha,L,N,\epsilon/3}$ is local. Thus, applying Lemma~\ref{continuite-des-boites} with $S=[-N,N]^d\cap \Zd$ and $t=\alpha L N$, we obtain the existence of $\lambda_1 \in (\lambda_0, \lambda_{\max}]$ such that 
$$\forall \lambda\in [\lambda_0,\lambda_1]\quad\P(A_{\lambda,\lambda_0}^{\alpha,L,N,\epsilon/3})>p_1.$$
Then, it follows that 
$$ \forall \lambda\in [\lambda_0,\lambda_1]\quad \exists A,B\quad  \forall t>0\quad \P(\xi^{0,\lambda}_t \not \subset (1+\epsilon)t\shape{\lambda_0})\le A\exp(-B t).$$
Now, we can deduce from \eqref{uniftau} and \eqref{richard} that
$$ \forall \lambda\in [\lambda_0,\lambda_1]\quad \exists A,B\quad  \forall t>0\quad \Pbarre_{\lambda}(H^{0,\lambda}_t \not \subset (1+\epsilon)t\shape{\lambda_0})\le A\exp(-B t).$$
A detailed proof is provided in~\cite{GM-contact-gd} when deducing (62) from (61).
Fix $\lambda\in[\lambda_0,\lambda_1]$, $\eta>0$ and, using the asymptotic shape result \eqref{leqdeforme}, choose $t$ large enough to have
$A\exp(-B t)<1/2$ and $\Pbarre_{\lambda}((1-\eta)t\shape{\lambda}  \not \subset H^{0,\lambda}_t)<1/2$.
This implies that the event $\{(1-\eta)t\shape{\lambda}\subset H^{0,\lambda}_t\subset (1+\epsilon)t\shape{\lambda_0}\}$ has positive probability; particularly, $(1-\eta)\shape{\lambda}\subset (1+\epsilon)\shape{\lambda_0}$, and, letting $\eta$ tend to $0$, we have

$$\forall \lambda\in[\lambda_0,\lambda_1] \quad \shape{\lambda}\subset (1+\epsilon)\shape{\lambda_0},$$
or equivalently
$$\forall \lambda\in[\lambda_0,\lambda_1] \quad\forall x\in\Rd\quad  \mu_{\lambda_0}(x)\le (1+\epsilon)\mu_{\lambda}(x).$$
 This completes the proof of \eqref{E:contdroite}.


\def\refname{References}
\bibliographystyle{plain}

\end{document}